\documentclass[12pt,a4paper,oneside,reqno,notitlepage]{amsart}
\usepackage{amssymb}
\textwidth
150mm

\newtheorem{theorem}{Theorem}
\theoremstyle{plain}

\newtheorem{lemma}{Lemma}

\numberwithin{equation}{section}

\begin{document}

\baselineskip 8mm
\parindent 9mm

\title[]
{Fractional elliptic equations with Hardy potential and critical nonlinearities}

\author{Kexue  Li}

\address{Kexue Li\newline
School of Mathematics and Statistics, Xi'an Jiaotong University, Xi'an 710049, China}
\email{kxli@mail.xjtu.edu.cn}

\thanks{{\it 2010 Mathematics Subjects Classification}: 35B33, 58E30}

\begin{abstract}
In this paper, we consider the fractional elliptic equation
\begin{align*}
\left\{\begin{aligned}
&(-\Delta)^{s}u-\mu\frac{u}{|x|^{2s}}=\frac{|u|^{2_{s}^{\ast}(\alpha)-2}u}{|x|^{\alpha}}+f(x,u), && \mbox{in} \ \Omega,\\
&u=0, && \mbox{in} \ \mathbb{R}^{n}\backslash \ \Omega,
\end{aligned}\right.
\end{align*}
where $\Omega\subset R^{n}$ is a smooth bounded domain, $0\in\Omega$, $0<s<1$, $0<\alpha<2s<n$, $2_{s}^{\ast}(\alpha)=\frac{2(n-\alpha)}{n-2s}$. Under some assumptions on $\mu$ and $f$, we obtain the existence of nonnegative solutions.
\end{abstract}
\maketitle

\section{\textbf{Introduction}}
In this paper, we are concerned with the following critical problem
\begin{align}\label{critical}
\left\{\begin{aligned}
&(-\Delta)^{s}u-\mu\frac{u}{|x|^{2s}}=\frac{|u|^{2_{s}^{\ast}(\alpha)-2}u}{|x|^{\alpha}}+f(x,u), && \mbox{in} \ \Omega,\\
&u=0, && \mbox{in}\ \mathbb{R}^{n}\backslash \ \Omega,
\end{aligned}\right.
\end{align}
where $\Omega\subset \mathbb{R}^{n}$ is a smooth bounded domain, $0\in\Omega$,  $0<s<1$, $0<\alpha<2s<n$, $0\leq \mu<\mu_{H}(s)=4^{s}\frac{\Gamma^{2}(\frac{n+2s}{4})}{\Gamma^{2}(\frac{n-2s}{4})}$, $2_{s}^{\ast}(\alpha)=\frac{2(n-\alpha)}{n-2s}$ is the fractional critical Hardy-Sobolev
exponent. The fractional Laplacian $(-\Delta)^{s}$ is defined by
\begin{align*}
(-\Delta)^{s}u(x)=c_{n,s} P.V.\int_{\mathbb{R}^{n}}\frac{u(x)-u(y)}{|x-y|^{n+2s}}dy,
\end{align*}
where $P.V.$ is the principal value and $c_{n,s}$ is a constant that depends on $n$ and $s$. If $s=1$, $(-\Delta)^{s}$ is defined as $-\Delta$, where $\Delta$ is the Laplacian.
In fact, for any $u\in C_{0}^{\infty}(\mathbb{R}^{n})$, $\lim_{s\rightarrow 1-}(-\Delta)^{s}u=-\Delta u$,
see Proposition 4.4 in \cite{DPV}. For $s=1$, there are many results for problem (\ref{critical}), we refer to \cite{Bhakta,Kang,Yan,Zhang}.

Recently, the existence of nontrivial solutions for nonlinear fractional elliptic equations with Hardy potential
\begin{align}\label{several}
(-\Delta)^{s}u-\mu\frac{u}{|x|^{2s}}=g(u) \ \ \ \ \mbox{in} \ \Omega,
\end{align}
have been studied by several authors. Barrios, Medina and Peral \cite{BMP} studied (\ref{several}) with $g(u)=u^{p}+\lambda u^{q}$ $(0<q<1,1<p<p(\mu,s))$ and discussed the existence and multiplicity of solutions depending on the value of $p$. Fall \cite{Fall} studied (\ref{several}) with $g(u)=u^{p}$ $(p>1)$ and obtained the existence and nonexistence of nonnegative distributional solutions. Shakerian \cite{Shakerian} studied the following problem
\begin{align}\label{Shakerian}
\left\{\begin{aligned}
&(-\Delta)^{s}u-\mu\frac{u}{|x|^{2s}}-\lambda u=\frac{|u|^{2_{s}^{\ast}(\alpha)-2}u}{|x|^{\alpha}}+h(x)u^{q-1}, && \mbox{in} \ \Omega,\\
&u=0, && \mbox{in} \ \mathbb{R}^{n}\backslash \ \Omega,
\end{aligned}\right.
\end{align}
where $\lambda\in \mathbb{R}$, $h\in C^{0}(\bar{\Omega})$, $h\geq 0$, $q\in (2,2_{s}^{*}(0))$. A sufficient condition is established for the existence of a positive solution for (\ref{Shakerian}). Zhang and Hsu \cite{ZH} investigated a system of fractional elliptic equation involving critical Sobolev-Hardy exponents and concave-convex nonlinearities, the existence and multiplicity of positive solutions was proved by variational methods. Motivated by the above papers, we will study the problem (\ref{critical}), here is the main result of this paper.

\begin{theorem}
Suppose that $0<s<1$, $0<\mu<\mu_{H}(s)$, $0<\alpha<2s<n$, $\frac{n-2s}{2}<\beta_{+}(\mu)<\frac{n}{2}$,
where $\beta_{+}(\mu)$ is the unique solution of $\Psi_{n,s}(\beta)=4^{s}\frac{\Gamma(\frac{n-\beta}{2})\Gamma(\frac{2s+\beta}{2})}{\Gamma(\frac{n-2s-\beta}{2})\Gamma(\frac{\beta}{2})}-\mu=0$ in $(\frac{n-2s}{2},n-2s]$, \\
$(f_{1})$ $f\in C(\Omega\times \mathbb{R},\mathbb{R})$ and there exist constants $c_{1},c_{2}>0$ and $p\in (2,2_{s}^{*}(0))$ such that $|f(x,t)|\leq c_{1}|t|+c_{2}|t|^{p-1}$ for all $x\in \Omega$, $t\in \mathbb{R}$,\\
$(f_{2})$ There exists a constant $K>0$ big enough such that $F(x,t)=\int_{0}^{t}f(x,s)ds\geq Kt^{2}$ for all $x\in \Omega$, $t\in \mathbb{R}$,\\
$(f_{3})$ There exist constants $\rho>2$, $\eta>0$ such that $\rho F(x,t)\leq f(x,t)t+\eta t^{2}$ for all $x\in \Omega$, $t\in \mathbb{R}$,\\
then problem (\ref{critical}) has at least a nonnegative solution.
\end{theorem}

\section{Preliminaries}
Let $\Omega$ be a smooth bounded domain in $\mathbb{R}^{n}$ with 0 in its interior.
For $p\in [1,\infty)$, we denote by $L^{p}(\Omega)$ the usual Lebesgue space with the norm $\|u\|_{p}=\left(\int_{\Omega}|u|^{p}dx\right)^{\frac{1}{p}}$. For $s\in (0,1)$, the fractional Sobolev space $H^{s}_{0}(\Omega)$ is defined as the closure of $C_{0}^{\infty}(\Omega)$ with respect to the norm
\begin{align*}
\|u\|^{2}_{H^{s}_{0}(\Omega)}=\int_{\mathbb{R}^{n}}|(-\Delta)^{\frac{s}{2}}u|^{2}dx=\frac{c_{n,s}}{2}\int_{\mathbb{R}^{n}}\int_{\mathbb{R}^{n}}\frac{|u(x)-u(y)|^{2}}{|x-y|^{n+2s}}dxdy,
\end{align*}
where $c_{n,s}=\frac{4^{s}\Gamma(\frac{n+2s}{2})}{\pi^{\frac{n}{2}}|\Gamma(-s)|}$.

For $0<\mu<\mu_{H}(s)$, define another norm
\begin{align}\label{another}
\||u\||=\left(\frac{c_{n,s}}{2}\int_{\mathbb{R}^{n}}\int_{\mathbb{R}^{n}}\frac{|u(x)-u(y)|^{2}}{|x-y|^{n+2s}}dxdy-\mu\int_{\Omega}\frac{|u|^{2}}{|x|^{2s}}dx\right)^{\frac{1}{2}}.
\end{align}
By the fractional Hardy inequality (see \cite{Frank})
\begin{align*}
\mu_{H}(s)\int_{\mathbb{R}^{n}}\frac{|u|^{2}}{|x|^{2s}}dx\leq \int_{\mathbb{R}^{n}}|(-\Delta)^{\frac{s}{2}}u|^{2}dx, \ u\in H_{0}^{s}(\mathbb{R}^{n}),
\end{align*}
we see that $\||\cdot\||$ is well defined and for $0<\mu<\mu_{H}(s)$, $\||\cdot\||$ is equalivalent to the norm $\|\cdot\|_{H^{s}_{0}(\Omega)}$.

In order to study positive solution for (\ref{critical}), we consider the existence of nontrivial solutions to the problem
\begin{align}\label{nonnegative}
(-\Delta)^{s}u-\mu\frac{u}{|x|^{2s}}=\frac{u_{+}^{2_{s}^{\ast}(\alpha)-1}}{|x|^{\alpha}}+f_{+}(x,u), \ \mbox{in} \ \Omega.
\end{align}
The energy functional $I: H_{0}^{s}(\Omega)\rightarrow \mathbb{R}$ whose critical points are weak solutions to problem (\ref{nonnegative}) is given by
\begin{align*}
I(u)=\frac{1}{2}\||u|\|^{2}-\frac{1}{2_{s}^{*}(\alpha)}\int_{\Omega}\frac{u_{+}^{2_{s}^{*}(\alpha)}}{|x|^{\alpha}}dx-\int_{\Omega}F_{+}(x,u)dx, \ u\in H_{0}^{s}(\Omega),
\end{align*}
where $u_{+}=\max\{u,0\}$, $F_{+}(x,t)=\int_{0}^{t}f_{+}(x,s)ds$,
$f_{+}(x,t)=\left\{\begin{aligned}
&f(x,t), && t\geq 0,\\
&0, && t<0.
\end{aligned}\right.$\\
The best constant in fractional Hardy-Sobolev inequality in $\mathbb{R}^{n}$ is
\begin{align}\label{Hardy-Sobolev}
\Lambda_{\mu,s,\alpha}=\inf_{u\in H^{s}(\mathbb{R}^{n})\backslash \{0\}}\frac{\int_{\mathbb{R}^{n}}(|(-\Delta)^{\frac{s}{2}}u|^{2}-\mu\frac{|u|^{2}}{|x|^{2s}})dx}{(\int_{\mathbb{R}^{n}}\frac{|u|^{2_{s}^{*}(\alpha)}}{|x|^{\alpha}}dx)^{\frac{2}{2_{s}^{*}(\alpha)}}}.
\end{align}
Ghoussoub and Shakerian \cite{GS} proved the existence of extremals for $\Lambda_{\mu,s,\alpha}(\mathbb{R}^{n})$, when $0<\alpha<2s<n$ and $-\infty<\mu<\mu_{H}(s)$. Note that any minimizer for (\ref{Hardy-Sobolev}) is a variational solution of the following borderline problem
\begin{align}\label{borderline}
\left\{\begin{aligned}
&(-\Delta)^{s}u-\mu\frac{u}{|x|^{2s}}=\frac{|u|^{2_{s}^{\ast}(\alpha)-2}u}{|x|^{\alpha}}, && \mbox{in} \ \mathbb{R}^{n},\\
&u\geq 0, u \not\equiv 0&& \mbox{in} \ \mathbb{R}^{n}.
\end{aligned}\right.
\end{align}
by \cite{GS,Shakerian}, for $0<\mu<\mu_{H}(s)$, the problem (\ref{borderline}) has positive radial symmetric solution $U_{\varepsilon}(x)=\varepsilon^{-(n-2s)/2}u_{\mu}(x/\varepsilon)$, where $u_{\mu}$ is a solution of (\ref{borderline}), $u_{\mu}\in C^{1}(\mathbb{R}^{n}\backslash \{0\})$, $u_{\mu}>0$. Let $\varphi\in C_{0}^{\infty}(\Omega)$, $0\leq \varphi(x)\leq 1$, choose $\rho>0$ small enough such that $\varphi(x)=1$ for $|x|<\rho/2$, $\varphi(x)=0$ for $|x|\geq \rho$. Set $u_{\varepsilon}(x)=\varphi(x)U_{\varepsilon}(x)$.

\begin{lemma}\emph{(\cite{ZH}).}
Assume that $0<s<1$, $n>2s$, $0\leq \mu<\mu_{H}(s)$, $0\leq \alpha<2s$, $1\leq q<2_{s}^{\ast}(\alpha)$. Then, as $\varepsilon\rightarrow 0$, we have the following estimates:
\begin{align}\label{estimate1}
\||u_{\varepsilon}\||^{2}=(\Lambda_{\mu,s,\alpha})^{\frac{n-\alpha}{2s-\alpha}}+O(\varepsilon^{2\beta_{+}(\mu)+2s-n}),
\end{align}
\begin{align}\label{estimate2}
\left(\int_{\Omega}\frac{|u_{\varepsilon}|^{2_{s}^{*}(\alpha)}}{|x|^{\alpha}}dx\right)^{\frac{2}{2_{s}^{*}(\alpha)}}=(\Lambda_{\mu,s,\alpha})^{\frac{n-\alpha}{2s-\alpha}}+O(\varepsilon^{2_{s}^{*}(\alpha)\beta_{+}(\mu)+\alpha-n}),
\end{align}
\begin{align}\label{estimate3}
\int_{\Omega}|u_{\varepsilon}|^{q}dx=\left\{\begin{aligned}
&O(\varepsilon^{n-\frac{q(n-2s)}{2}}), && q>\frac{n}{\beta_{+}(\mu)},\\
&O(\varepsilon^{n-\frac{q(n-2s)}{2}}|\log \varepsilon|), && q=\frac{n}{\beta_{+}(\mu)},\\
&O(\varepsilon^{q(\beta_{+}(\mu)-\frac{n-2s}{2})}), && q<\frac{n}{\beta_{+}(\mu)},
\end{aligned}\right.
\end{align}
where $\beta_{+}(\mu)$ is the unique solution of $\Psi_{n,s}(\beta)=4^{s}\frac{\Gamma(\frac{n-\beta}{2})\Gamma(\frac{2s+\beta}{2})}{\Gamma(\frac{n-2s-\beta}{2})\Gamma(\frac{\beta}{2})}-\mu=0$ in $(\frac{n-2s}{2},n-2s]$.
\end{lemma}

\section{Proof of Theorem \ref{critical}}
\begin{lemma}\label{geometry}
Suppose that $(f1)$, $(f2)$, $(f3)$ hold. Let $\{u_{n}\}\subset H_{0}^{s}(\Omega)$ be a sequence such that $I(u_{n})\rightarrow c$ and $I'(u_{n})\rightarrow 0$ in $(H_{0}^{s}(\Omega))^{-1}$, where $c\in \left(0,\frac{2s-\alpha}{2(n-\alpha)}(\Lambda_{\mu,s,\alpha})^{\frac{n-\alpha}{2s-\alpha}}\right)$. Then there exists $u\in H_{0}^{s}(\Omega)$ such that $u_{n}\rightharpoonup u$, up to a subsequence, $I'(u)=0$ and $u$ is a nontrival solution of problem (\ref{critical}).
\end{lemma}
\begin{proof}
We use the method of contradiction to show that $\{u_{n}\}$ is bounded. Suppose that $\||u_{n}\||\rightarrow \infty$ as $n\rightarrow \infty$. Let $v_{n}=u_{n}/\||u_{n}\||$, then $\||v_{n}\||=1$. By $(f_{3})$, we get
\begin{align}\label{bounded}
c+1+o(1)\||u_{n}\||&\geq I(u_{n})-\frac{1}{\theta}\langle I'(u_{n}),u_{n}\rangle\nonumber\\
&=\left(\frac{1}{2}-\frac{1}{\theta}\right)\||u_{n}\||^{2}+\left(\frac{1}{\theta}-\frac{1}{2_{s}^{*}(0)}\right)\int_{\Omega}\frac{(u_{n})_{+}^{2_{s}^{*}(\alpha)}}{|x|^{\alpha}}dx\nonumber\\
&\quad-\int_{\Omega}\left(F_{+}(x,u_{n})-\frac{1}{\theta}f_{+}(x,u_{n})u_{n}\right)dx\nonumber\\
&\geq \left(\frac{1}{2}-\frac{1}{\theta}\right)\||u_{n}\||^{2}-\frac{\eta}{\theta}\|u_{n}\|_{2}^{2}\nonumber\\
&=\||u_{n}\||^{2}\left(\frac{\theta-2}{2\theta}-\frac{\eta}{\theta}\|v_{n}\|_{2}^{2}\right),
\end{align}
where $\theta=\min\{2_{s}^{*}(\alpha),\rho\}$. Up to a subsequence, we assume that $v_{n}\rightharpoonup v$ in $H_{0}^{s}(\Omega)$. By Corollary 7.2 in \cite{DPV}, $H_{0}^{s}(\Omega)\hookrightarrow L^{q}(\Omega)$ is compact for
$1\leq q<2_{s}^{\ast}(0)$. Then $v_{n}\rightarrow v$ in $L^{q}(\Omega)$, $1\leq q<2_{s}^{\ast}(0)$, and  $v_{n}\rightarrow v$ \mbox{a.e.} in $\Omega$. Thus, by (\ref{bounded}), we get $v\neq 0$. On the other hand,
\begin{align*}
0&=\lim_{n\rightarrow \infty}\frac{I(u_{n})}{\||u_{n}\||^{2}}\\
&\quad\leq \lim_{n\rightarrow \infty}\left(\frac{1}{2}-\int_{\Omega}\frac{F_{+}(x,u_{n})}{u_{n}^{2}}v_{n}^{2}dx\right)\\
&\quad\leq \lim_{n\rightarrow \infty}\left(\frac{1}{2}-K\int_{\Omega}v_{n}^{2}dx\right)\\
&\quad=\frac{1}{2}-K\int_{\Omega}v^{2}dx<0,
\end{align*}
since $K$ is big enough, this is a contradiction. Therefore $\{u_{n}\}$ is bounded in $H_{0}^{s}(\Omega)$ and there exists $u\in H_{0}^{s}(\Omega)$ such that $u_{n}\rightharpoonup u$ up to a subsequence. By the weak continuity of
$I'$, we have $I'(u)=0$. Assume that $u=0$ in $\Omega$, since $f(x,u)$ is subcritical, from $\langle I'(u_{n}),u_{n}\rangle=o(1)$,
\begin{align}\label{alpha}
\||u_{n}\||^{2}-\int_{\Omega}\frac{(u_{n})_{+}^{2_{s}^{*}(\alpha)}}{|x|^{\alpha}}dx=o(1).
\end{align}
By the definition of $\Lambda_{\mu,s,\alpha}$,
\begin{align}\label{Lambda}
\||u_{n}\||^{2}\geq \Lambda_{\mu,s,\alpha}\left(\int_{\Omega}\frac{|u_{n}|^{2_{s}^{*}(\alpha)}}{|x|^{\alpha}}dx\right)^{\frac{2}{2_{s}^{*}(\alpha)}}.
\end{align}
By (\ref{alpha}) and (\ref{Lambda}),
\begin{align}\label{musalpha}
o(1)\geq \||u_{n}\||^{2}\left(1-(\Lambda_{\mu,s,\alpha})^{-\frac{2_{s}^{*}(\alpha)}{2}}\||u_{n}\||^{2_{s}^{*}(\alpha)-2}\right).
\end{align}
If $\||u_{n}\||\rightarrow 0$, this contradict $c>0$. Then by (\ref{musalpha}),
\begin{align}\label{lower}
\||u_{n}\||^{2}\geq (\Lambda_{\mu,s,\alpha})^{\frac{n-\alpha}{2s-\alpha}}+o(1).
\end{align}
It follows from (\ref{alpha}) and (\ref{lower}) that
\begin{align*}
I(u_{n})&=\frac{1}{2}\||u_{n}\||^{2}-\frac{1}{2_{s}^{*}(\alpha)}\int_{\Omega}\frac{(u_{n})_{+}^{2_{s}^{*}(\alpha)}}{|x|^{\alpha}}dx+o(1)\\
&=\frac{2s-\alpha}{2(n-\alpha)}\||u_{n}\||^{2}+o(1)\\
&\geq \frac{2s-\alpha}{2(n-\alpha)}(\Lambda_{\mu,s,\alpha})^{\frac{n-\alpha}{2s-\alpha}}+o(1),
\end{align*}
this contradicts $c<\frac{2s-\alpha}{2(n-\alpha)}(\Lambda_{\mu,s,\alpha})^{\frac{n-\alpha}{2s-\alpha}}$. Therefore $u\neq 0$ and $u$ is a nontrival solution of problem (\ref{critical}).
\end{proof}
\begin{lemma}\label{energy}
Assume that $0<\alpha<2s<n$, $0\leq \mu<\mu_{H}(s)$, $\frac{n-2s}{2}<\beta_{+}(\mu)<\frac{n}{2}$ and $(f_{1})$, $(f_{2})$ hold. Then there exists $u_{0}\in H_{0}^{s}(\Omega)$, $u_{0}\neq 0$, such that
\begin{align*}
\sup_{t\geq 0}I(tu_{0})<\frac{2s-\alpha}{2(n-\alpha)}(\Lambda_{\mu,s,\alpha})^{\frac{n-\alpha}{2s-\alpha}}.
\end{align*}
\end{lemma}
\begin{proof}
Since $U_{\varepsilon}(x)>0$, then $u_{\varepsilon}(x)\geq 0$ and $(u_{\varepsilon})_{+}=\max\{u_{\varepsilon},0\}=u_{\varepsilon}$. Set
\begin{align*}
g(t)=I(tu_{\varepsilon})=\frac{t^{2}}{2}\||u_{\varepsilon}\||^{2}-\frac{t^{2_{s}^{*}(\alpha)}}{2_{s}^{*}(\alpha)}\int_{\Omega}\frac{u_{\varepsilon}^{2_{s}^{*}(\alpha)}}{|x|^{\alpha}}dx-\int_{\Omega}F_{+}(x,tu_{\varepsilon})dx.
\end{align*}
By (\ref{estimate2}),
\begin{align*}
g(t)=\frac{t^{2}}{2}\||u_{\varepsilon}\||^{2}-\frac{t^{2_{s}^{*}(\alpha)}}{2_{s}^{*}(\alpha)}\left[(\Lambda_{\mu,s,\alpha})^{\frac{n-\alpha}{2s-\alpha}}+O(\varepsilon^{2_{s}^{*}(\alpha)\beta_{+}(\mu)+\alpha-n})\right]
-\int_{\Omega}F_{+}(x,tu_{\varepsilon})dx.
\end{align*}
Let
\begin{align*}
\bar{g}(t)=\frac{t^{2}}{2}\||u_{\varepsilon}\||^{2}-\frac{t^{2_{s}^{*}(\alpha)}}{2_{s}^{*}(\alpha)}(\Lambda_{\mu,s,\alpha})^{\frac{n-\alpha}{2s-\alpha}}.
\end{align*}
Note that $\lim_{t\rightarrow +\infty}g(t)=-\infty$, $g(0)=0$, $g(t)>0$ for $t>0$ small enough, thus $\sup_{t\geq 0}g(t)$ is attained for some $t_{\varepsilon}>0$. Since
\begin{align*}
0&=g'(t_{\varepsilon})\\
&=t_{\varepsilon}\left(\||u_{\varepsilon}\||^{2}-t_{\varepsilon}^{2_{s}^{*}(\alpha)-2}[(\Lambda_{\mu,s,\alpha})^{\frac{n-\alpha}{2s-\alpha}}+O(\varepsilon^{2_{s}^{*}(\alpha)\beta_{+}(\mu)+\alpha-n})]
-\frac{1}{t_{\varepsilon}}\int_{\Omega}f_{+}(x,t_{\varepsilon}u_{\varepsilon})u_{\varepsilon}dx\right),
\end{align*}
we have
\begin{align*}
\||u_{\varepsilon}\||^{2}&=t_{\varepsilon}^{2_{s}^{*}(\alpha)-2}[(\Lambda_{\mu,s,\alpha})^{\frac{n-\alpha}{2s-\alpha}}+O(\varepsilon^{2_{s}^{*}(\alpha)\beta_{+}(\mu)+\alpha-n})]+\frac{1}{t_{\varepsilon}}\int_{\Omega}f_{+}(x,t_{\varepsilon}u_{\varepsilon})u_{\varepsilon}dx\\
&\geq (\Lambda_{\mu,s,\alpha})^{\frac{n-\alpha}{2s-\alpha}}t_{\varepsilon}^{2_{s}^{*}(\alpha)-2},
\end{align*}
then
\begin{align*}
t_{\varepsilon}\leq \left[\frac{\||u_{\varepsilon}\||^{2}}{(\Lambda_{\mu,s,\alpha})^{\frac{n-\alpha}{2s-\alpha}}}\right]^{\frac{1}{2_{s}^{*}(\alpha)-2}}\triangleq t_{\varepsilon}^{0}.
\end{align*}
By $(f_{1})$, we have
\begin{align}\label{f}
\||u_{\varepsilon}\||^{2}&=t_{\varepsilon}^{2_{s}^{*}(\alpha)-2}[(\Lambda_{\mu,s,\alpha})^{\frac{n-\alpha}{2s-\alpha}}+O(\varepsilon^{2_{s}^{*}(\alpha)\beta_{+}(\mu)+\alpha-n})]+c_{1}\int_{\Omega}|u_{\varepsilon}|^{2}dx\nonumber\\
&\quad+c_{2}t_{\varepsilon}^{p-2}\int_{\Omega}|u_{\varepsilon}|^{p}dx.
\end{align}
Choosing $\varepsilon$ small enough, by (\ref{f}), (\ref{estimate1}) and (\ref{estimate3}), we have
\begin{align}\label{same}
t_{\varepsilon}^{2_{s}^{*}(\alpha)-2}\geq \frac{1}{2}.
\end{align}
On the other hand, $\bar{g}(t)$ attains its maximum at $t_{\varepsilon}^{0}=\left[\frac{\||u_{\varepsilon}\||^{2}}{(\Lambda_{\mu,s,\alpha})^{\frac{n-\alpha}{2s-\alpha}}}\right]^{\frac{1}{2_{s}^{*}(\alpha)-2}}$ and is increasing in $[0,t_{0}]$. Since $\frac{n-2s}{2}<\beta_{+}(\mu)<\frac{n}{2}$, by (\ref{estimate1}), (\ref{estimate3}), (\ref{same}) and $(f_{2})$,
\begin{align*}
g(t_{\varepsilon})&\leq \bar{g}(t_{0})-\int_{\Omega}F_{+}(x,t_{\varepsilon}u_{\varepsilon})dx\\
&=\frac{2s-\alpha}{2(n-\alpha)(\Lambda_{\mu,s,\alpha})^{\frac{(n-\alpha)(n-2s)}{(2s-\alpha)^{2}}}}\||u_{\varepsilon}\||^{\frac{2(n-\alpha)}{2s-\alpha}}-\int_{\Omega}F_{+}(x,t_{\varepsilon}u_{\varepsilon})dx\\
&\leq \frac{2s-\alpha}{2(n-\alpha)}(\Lambda_{\mu,s,\alpha})^{\frac{n-\alpha}{2s-\alpha}}+O(\varepsilon^{2\beta_{+}(\mu)+2s-n})-Kt_{\varepsilon}^{2}\int_{\Omega}|u_{\varepsilon}|^{2}dx\\
&\leq  \frac{2s-\alpha}{2(n-\alpha)}(\Lambda_{\mu,s,\alpha})^{\frac{n-\alpha}{2s-\alpha}}+O(\varepsilon^{2\beta_{+}(\mu)+2s-n})-K\left(\frac{1}{2}\right)^{\frac{2}{2_{s}^{*}(\alpha)-2}}O(\varepsilon^{2\beta_{+}(\mu)+2s-n}).
\end{align*}
Since $K$ is big enough, we have
\begin{align*}
\sup_{t\geq 0}I(tu_{\varepsilon})=g(t_{\varepsilon})<\frac{2s-\alpha}{2(n-\alpha)}(\Lambda_{\mu,s,\alpha})^{\frac{n-\alpha}{2s-\alpha}}.
\end{align*}
\end{proof}
\textbf{Proof of Theorem \ref{critical}.}
From fractional Hardy-Sobolev inequality and fractional Sobolev inequality, it follows that
\begin{align}\label{Hardy}
\int_{\mathbb{R}^{n}}\frac{|u|^{2_{s}^{*}(\alpha)}}{|x|^{\alpha}}dx\leq C\||u\||^{2_{s}^{*}(\alpha)}, \ \|u\|_{q}^{q}\leq C\||u\||^{q}, \ 1\leq q\leq 2_{s}^{*}(0), \ u\in H_{0}^{s}(\Omega).
\end{align}
By (\ref{Hardy}), $(f_{1})$,
\begin{align*}
I(u)&=\frac{1}{2}\||u|\|^{2}-\frac{1}{2_{s}^{*}(\alpha)}\int_{\Omega}\frac{u_{+}^{2_{s}^{*}(\alpha)}}{|x|^{\alpha}}dx-\int_{\Omega}F_{+}(x,u)dx\\
&\geq \frac{1}{2}\||u|\|^{2}-\frac{C}{2_{s}^{*}(\alpha)}\||u_{+}\||^{2_{s}^{*}(\alpha)}-\frac{Cc_{1}}{2}\|u\|_{2}^{2}-\frac{Cc_{2}}{p}\|u\|_{p}^{p}\\
&\geq \frac{1-Cc_{1}}{2}\||u|\|^{2}-\frac{C}{2_{s}^{*}(\alpha)}\||u_{+}\||^{2_{s}^{*}(\alpha)}-\frac{C}{p}\|u\|_{p}^{p}.
\end{align*}
Thus, there exists $\delta>0$ such that $I(u)\geq \delta$ for all $u\in\partial B_{r}(0)=\{u\in H_{0}^{s}(\Omega),\||u|\|=r\}$, where $r>0$ small enough. Since $F_{+}(x,t)\geq 0$, for $u_{0}\neq 0$, we have
\begin{align*}
I(tu_{0})\leq \frac{t^{2}}{2}\||u_{0}|\|^{2}-\frac{t^{2_{s}^{*}(\alpha)}}{2_{s}^{*}(\alpha)}\int_{\Omega}\frac{u_{+}^{2_{s}^{*}(\alpha)}}{|x|^{\alpha}}dx,
\end{align*}
then $\lim_{t\rightarrow +\infty}I(tu_{0})\rightarrow -\infty$. We can choose $t_{0}$ such that $\|t_{0}u_{0}\|>r$ and $I(t_{0}u_{0})<0$. By the Mountain Pass Lemma \cite{Rabinowitz}, there is a sequence $\{u_{n}\}\subset H_{0}^{s}(\Omega)$ satisfying
\begin{align*}
I(u_{n})\rightarrow c\geq \delta,  \ I'(u_{n})\rightarrow 0,
\end{align*}
where
\begin{align*}
c&=\inf_{\gamma\in \Gamma}\max_{t\in[0,1]}I(\gamma(t)),\\
\Gamma&=\{\gamma(t)\in C([0,1],H_{0}^{s}(\Omega))|\gamma(0)=0,\ \gamma(1)=t_{0}u_{0}\}.
\end{align*}
By Lemma \ref{geometry} and Lemma \ref{energy}, we get a $(PS)_{c}$ sequence $\{u_{n}\}\subset H_{0}^{s}(\Omega)$ and $u\in H_{0}^{s}(\Omega)$ such that $I'(u)=0$. Then $u$ is a solution of (\ref{nonnegative}) and
$\langle I'(u), u^{-}\rangle=0$, where $u^{-}=\min\{u,0\}$. Therefore $u\geq 0$, we get a nonnegative solution of (\ref{critical}).

\section{Acknowledgements}
This work is partially supported by NSFC under the grant 11571269, China Postdoctoral Science Foundation Funded Project under grants 2015M572539, 2016T90899.


\end{document}